\documentclass[reqno,11pt]{article} %

\usepackage{hyperref} 
\usepackage{suffix}
\usepackage{mathtools}
\usepackage{amsmath}
\usepackage{amsfonts}
\usepackage{amssymb}
\usepackage{graphicx}
\usepackage{epstopdf}
\usepackage[T1]{fontenc}
\usepackage[inner=2.5cm, outer=2.5cm, bottom=2.5cm, top=2.5cm]{geometry}
\usepackage{titlesec}
\usepackage{float}

\mathtoolsset{showonlyrefs}
\newcommand*{\spb}{\begin{split}}
\newcommand*{\spe}{\end{split}}
\newcommand*{\al}{\alpha}

\newcommand{\abs}[1]{\left| #1 \right|}

\newcommand{\R}{\mathbb{R}}

\newcommand{\pr}{\mathbf{P}}
\newcommand{\ex}{\mathbf{E}}
\newcommand{\N}{\mathbb{N}}
\newcommand{\ind}{\mathbf{1}}

\newcommand{\fal}{{\frac{1}{\al}}}

\newtheorem{theorem}{Theorem}[section]

\newtheorem{corollary}{Corollary}[section]
\newenvironment{proof}{\par\noindent {\bf Proof.}}
{\begin{flushright} \vspace*{-6mm}\mbox{$\Box$} \end{flushright}}

\makeatletter
    \@addtoreset{equation}{section} 
\makeatother

\DeclarePairedDelimiterX\MeijerM[3]{\lparen}{\rparen}%
{#3\delimsize\vert\,\begin{smallmatrix}#1 \\ #2\end{smallmatrix}}

\newcommand\MeijerG[8][]{%
  G^{\,#2,#3}_{#4,#5}\MeijerM[#1]{#6}{#7}{#8}}

\WithSuffix\newcommand\MeijerG*[7]{%
  G^{\,#1,#2}_{#3,#4}\MeijerM*{#5}{#6}{#7}}
\newcommand*{\eqb}{\begin{equation}}
\newcommand*{\eqe}{\end{equation}}

\title{Densities of L\'evy walks and the corresponding fractional equations}
 \author{\normalsize Marcin Magdziarz $^1$, Tomasz Zorawik $^1$}
\date{}
 
\begin{document}
 \maketitle

 \begin{abstract}

In this paper we derive explicit formulas for the densities of L\'evy walks.
Our results cover both jump-first and wait-first scenarios.
The obtained densities solve certain fractional differential
equations involving fractional material derivative operators.
In the particular case, when the stability index is rational,
the densities can be represented as an integral of Meijer G function.
This allows to efficiently evaluate them numerically.
Our results show perfect agreement with the Monte Carlo simulations.

\end{abstract}

\bigskip

\section{Introduction}\label{intro}
The Continuous Time Random Walk (CTRW) is a stochastic process determined uniquely by i.i.d.\ random variables
${J_1}$, ${J_2},\dots$ representing consecutive jumps of the random walker,
and i.i.d. positive random variables $T_1,T_2,\dots$ representing waiting times between jumps, \cite{Sokolov_book}.
The trajectories of CTRW are step functions with intervals $T_i$ and jumps $J_i$.

\emph{L\'evy walk} is a particular case of CTRW satisfying the following additional condition $|J_i| = T_i$ for every $i\in\mathbb{N}$ (length of the jump equal to the length of the preceding waiting time).
L\'evy walks were defined for the first time in \cite{Klafter1,Klafter2} in the framework
of generalized master equations. Since then, they became
one of the most popular models in statistical physics with large number of important applications.
L\'evy walks have been found to be excellent models in the description of various real-life phenomena and complex anomalous systems. The most striking examples are: transport of light in optical materials \cite{Wiersma},
foraging patterns of animals \cite{Bell,Berg,Buchanan}, epidemic spreading \cite{Brockmann2,Dybiec},
human travel \cite{Brockmann,Gonzales}, blinking nano-crystals \cite{Margolin}, and
fluid flow in a rotating annulus \cite{Solomon}.

The main idea underlying the L\'evy walk is the spatial-temporal
coupling, which is manifested by the condition $|J_i| = T_i$.
Thus, even if we assume that the distribution of jumps $Y_i$ is
heavy-tailed with diverging moments, the L\'evy walk itself has
finite moments of all orders (intuitively, long jumps are
penalized by requiring more time to be performed). This is very
different from $\alpha$-stable L\'evy processes with $\alpha<2$,
which have infinite second moment. These desirable properties make
L\'evy walk particularly attractive for physical applications.

Let us now recall the formal definition of L\'evy walk.
Let $T_i$, $i=1,2,\ldots$, be the sequence of i.i.d.
positive random variables, representing the waiting times of the walker. Assume that they belong to the domain of attraction of one-sided $\al$-stable law, i.e.
\begin{equation} \label{assumption_1}
n^{-1/\al}\sum_{i=1}^{[nt]}T_i \stackrel{d}{\longrightarrow} S_\al(t)
\end{equation}
as $n\rightarrow\infty$. Here, $S_\al(t)$ is the $\al$-stable subordinator
with the Laplace transform \cite{JW94}
\[
\mathbb{E}(\exp\{-sS_\al(t)\})=\exp\{-t s^\al\}, \;\;\; 0<\al<1.
\]
Denote by
\[
N(t)=\max\{k\geq 0:\sum_{i=1}^k T_i \leq t\}
\]
the corresponding counting process. Next,
define the jumps of the walker as
\eqb
\label{Jumps}
J_i= I_i T_i,
\eqe
where $I_1,I_2,\ldots$ are i.i.d. random variables,
which are assumed independent of the sequence of waiting times $T_i$.
Each $I_i$ governs the direction of the jump, i.e.
\eqb
\label{Direction}
\mathbb{P}(I_i=1)=p\;,\;\;\;\;\;\mathbb{P}(I_i=-1)=1-p,
\eqe
$0\leq p\leq1$. Thus, the walker jumps up with probability $p$ and down
with probability $1-p$. Clearly, the condition $|J_i| = T_i$ is satisfied.
The sequence of jumps $J_i$ defined above belongs to the domain of attraction
of $\al$-stable distribution \cite{MSZ}
\[
n^{-1/\al}\sum_{i=1}^{[nt]}J_i \stackrel{d}{\longrightarrow} L_\al(t)
\]
as $n\rightarrow\infty$. The process $L_\al(t)$ is the $\al$-stable L\'evy motion
with the corresponding Fourier transform
\begin{eqnarray}
&&\mathbb{E}(\exp\{ikL_\al(t) \})  \label{char_fun_L} \\
\nonumber &&=\exp\{-t|k|^\al \cos(\pi \al /2)(1-i(2p-1)\tan(\pi \al /2 )sgn(k))\}.
\end{eqnarray}
Later on we will also use the corresponding left-limit process
\eqb
\label{left_limit}
L^-_\al(t)\stackrel{def}{=}\lim_{s\nearrow t}L_\al(s).
\eqe
Finally, the process
\[
R(t)=\sum_{i=1}^{N(t)}J_i
\]
is called L\'evy walk. $R(t)$ is also known as \emph{wait-first L\'evy walk} in the literature \cite{LW_review},
since the walker at the beginning of its motion ($t=0$) first waits and then performs the jump.

As shown in \cite{MSZ}, $R(t)$ obeys the following scaling limit
in distribution \eqb \label{limit1}
\frac{R(nt)}{n}\stackrel{d}{\longrightarrow} X(t) \eqe as
$n\rightarrow\infty$. Here $X(t)$ is the right-continuous version
of the process $L^-_\al (S^{-1}_\al(t))$. Moreover,
$S^{-1}_\al(t)$ is the inverse of $S_\al(t)$, i.e.
$S_\al^{-1}(t)=\inf\{\tau\geq 0:S_\al(\tau)>t\}$. Additionally,
the instants of jumps as well as the respective jump lengths of
the processes $L_\al(t)$ and $S_\al(t)$ are exactly the same.
Also, the probability density function (PDF) $p_t(x)$ of $X(t)$
satisfies the following fractional equation \cite{Jurlewicz,MSZ}
\eqb \label{FDE1} \left [ p\left( \frac{\partial}{\partial t} -
\frac{\partial}{\partial x} \right)^\al + (1-p)\left(
\frac{\partial}{\partial t} + \frac{\partial}{\partial x}
\right)^\al\right] p_t(x) =
\delta_0(x)\frac{t^{-\al}}{\Gamma(1-\al)}. \eqe Here, the
operators $\left( \frac{\partial}{\partial t} \mp
\frac{\partial}{\partial y} \right)^\al$ are the fractional
material derivatives introduced in \cite{Sokolov_Metzler}. In the
Fourier-Laplace space they are given by
\[
\mathcal{F}_y\mathcal{L}_t\left\{\left( \frac{\partial}{\partial t} \mp \frac{\partial}{\partial y} \right)^\al f(y,t)\right\}=(s\mp ik)^\al f(k,s).
\]
In what follows, we will find the explicit solution of \eqref{FDE1} by
determining the PDF $p_t(x)$ of the limit process $X(t)$ given in \eqref{limit1}.

We will also consider the so-called \emph{jump-first L\'evy walk} \cite{MTZ}
\[
\widetilde{R}(t)=\sum_{i=1}^{N(t)+1}J_i .
\]
It has the following scaling limit in distribution \cite{MT,MTZ}
\eqb \label{limit2}
\frac{\widetilde{R}(nt)}{n}\stackrel{d}{\longrightarrow}
Y(t)=L_\al(S^{-1}_\al(t)) \eqe as $n\rightarrow\infty$. Moreover,
the PDF $w_t(y)$ of $Y(t)$ satisfies the fractional equation of
the form \cite{Jurlewicz,MT} \eqb \label{FDE2} \left [ p\left(
\frac{\partial}{\partial t} - \frac{\partial}{\partial y}
\right)^\al + (1-p)\left( \frac{\partial}{\partial t} +
\frac{\partial}{\partial y} \right)^\al\right] w_t(y) =
\frac{\al}{\Gamma(1-\al)}\int_t^\infty \delta_0 (y-u)u^{-\al-1}
du. \eqe

In the next two sections we will derive explicit formulas for the
PDFs of the limit processes $X(t)$ and $Y(t)$, respectively. This
way we will obtain solutions of fractional equations \eqref{FDE1}
and \eqref{FDE2}.

\section{Densities of wait-first L\'evy walks}
Let us start with the wait-first scenario. In the result below we
determine the PDF of the limit processes $X(t)$ from eq.
\eqref{limit1}.
\begin{theorem}\label{T1}
If $p\in(0,1)$ and $\al \in (0,1)$ then the PDF $p_t(x)$ of the process $X(t)=L_\al^-(S_\al^{-1}(t))$ equals
\eqb
\label{T1_thesis}
\begin{split}
p_t(x)&=\frac{1}{2p^\fal(1-p)^\fal}\frac{\al}{\Gamma (1-\al)}\\
&\quad\quad\int_{|x|}^t\int_0^\infty \frac{z^{1-\al}}{(t-w)^\al}r_\al\left(\frac{w+x}{2p^{\frac{1}{\al}}}z\right)r_\al\left(\frac{w-x}{2(1-p)^{\frac{1}{\al}}}z\right)dzdw \ind_{(0,t)}(|x|),
\end{split}
\eqe
where $r_\al(x)$ is a density of a positive $\al$-stable random variable $Z_\al$ with the Laplace transform
\eqb
\ex e^{-uZ_\al}=e^{-u^\al}.
\eqe
\end{theorem}
%
%
\begin{proof}
We start with reminding the following equality of distributions (see \cite{MSZ}):
\eqb
\label{distributions_equality}
\begin{split}
&\left(L_\al(t),S_\al(t)\right) \\
& \quad\quad\stackrel{d}= \left(p^{\frac{1}{\al}} S_\al^{(1)}(t)-(1-p)^{\frac{1}{\al}}S_\al^{(2)}(t), p^{\frac{1}{\al}} S_\al^{(1)}(t)+(1-p)^{\frac{1}{\al}}S_\al^{(2)}(t)\right),
\end{split}
\eqe
where $S_\al^{(1)}(t)$ and $S_\al^{(2)}(t)$ are independent copies of $S_\al(t)$. Using this fact it was shown in \cite{MSZ} that the L\'evy measure $\nu_{\left(L_\al,S_\al\right)}$ of $(L_\al(t), S_\al(t))$ equals
\eqb
\label{levy_measure}
\nu(dx,ds)=p\delta_{s}(dx)\nu_{S_\al}(ds)+(1-p)\delta_{-s}(dx)\nu_{S_\al}(ds),
\eqe
where $\nu_{S_\al}(ds)=\frac{\al}{\Gamma(1-\al)}x^{-1-\al}ds$. Hence, an infinitesimal generator $A$ of the process $(L_\al(t), S_\al(t))$ equals
\eqb
\begin{split}
\label{generator}
Af(x,s)&=\int_{\R^2}\left[f(x+y,s+w)-f(x,s)\right]\nu_{(L_\al,S_\al)}(dy,dw)\\
&=\int_{\R^2}\left[f(x+y,s+w)-f(x,s)\right]\\
&\quad\quad\quad\quad\quad\quad\left(p\delta_{w}(dy)\nu_{S_\al}(dw)+(1-p)\delta_{-w}(dy)\nu_{S_\al}(dw)\right).
\end{split}
\eqe In \cite{MS} the authors consider (among other processes) the
2-dimensional process $\left(X(t),V(t-)\right)$, where  $V(t)$ is
the age process counting time that passed since the last jump of
$X(t)$. Equation \eqref{generator} together with Remark 4.2 in the
mentioned paper provides us with the joint distribution of
$(X(t),V(t-))$
\eqb \label{xv_distribution} \pr
\left(X(t)=dx,V(t-)=dv\right)=\nu_{(L_\al,S_\al)}(\R\times[v,\infty))U(dx,t-dv)\ind_{\{0\leq
v \leq t\}}, \eqe where $U(dx,ds)$ is the 0-potential measure of
the process $(L_\al(t),S_\al(t))$, defined as \eqb
U(dx,ds)=\int_0^\infty \pr \left(L_\al(u)=dx,S_\al(u)=ds\right)du.
\eqe The process $V(t)$ is beyond our interest, but we will later
obtain the distribution of $X(t)$ as a marginal distribution of
$(X(t),V(t-))$. This is the reason why we turn now our  attention
to Eq.\eqref{xv_distribution}. We have \eqb
\begin{split}
\label{levy_measure_calculated}
\nu_{(L_\al,S_\al)}(\R\times[v,\infty))&=p\nu_{S_\al}([v,\infty))+(1-p)\nu_{S_\al}([v,\infty))=\nu_{S_\al}([v,\infty))\\
&=\int_v^\infty \frac{\al}{\Gamma (1-\al)}x^{-1-\al}dx=\frac{v^{-\al}}{\Gamma (1-\al)}.
\end{split}
\eqe
Furthermore, a density $u(x,s)$ of the potential measure $U$ can be calculated as
\eqb
\label{potential_density_def}
u(x,s)= \int_0^\infty w_{u}(x,s)du,
\eqe
where $w_{t}(x,s)$ is the PDF of the process $(L_\al(t),S_\al(t))$. From Eq. \eqref{distributions_equality} we get
\eqb
\begin{split}
&\pr\left(L_\al(t)= dx, S_\al(t)=ds\right)
=
\pr\left(S_\al^{(1)}=\frac{ds+dx}{2p^{\frac{1}{\al}}}, S_\al^{(2)}=\frac{ds-dx}{2(1-p)^{\frac{1}{\al}}}\right).
\end{split}
\eqe
The Jacobian determinant in the above formula for the linear transformation of $(dx,ds)$ equals $\frac{1}{2p^{1/\al}(1-p)^{1/\al}}$. Taking into account the independence of $S_\al^{(1)}(t)$ and $S_\al^{(2)}(t)$ we express $w_t(x,s)$ in terms of $q_t(s)$ - the PDF of $S_\al(t)$:
\eqb
w_t(x,s)=\frac{1}{2p^{\frac{1}{\al}}(1-p)^{\frac{1}{\al}}}q_t\left(\frac{s+x}{2p^{\frac{1}{\al}}}\right)q_t\left(\frac{s-x}{2(1-p)^{\frac{1}{\al}}}\right).
\eqe
Moreover the self-similarity of $S_\al(t)$ provides us with the equation
\eqb
q_u(x)=\frac{1}{u^{\frac{1}{\al}}}q_1\left(\frac{x}{u^\frac{1}{\al}}\right)=\frac{1}{u^{\frac{1}{\al}}}r_\al\left(\frac{x}{u^\frac{1}{\al}}\right),
\eqe
here $r_\al(x)$ is the density of the random variable $S_\al(1)$. We take into account the two above equations in Eq.\eqref{potential_density_def} and substitute $u=z^{-\al}$ to calculate the integral:
\eqb
\begin{split}
\label{potential_density2}
&u(x,s)=\int_0^\infty \frac{1}{2p^{\frac{1}{\al}}(1-p)^{\frac{1}{\al}}} q_u\left(\frac{s+x}{2p^{\frac{1}{\al}}}\right)q_u\left(\frac{s-x}{2(1-p)^{\frac{1}{\al}}}\right)du\\
&=\int_0^\infty \frac{1}{2p^{\frac{1}{\al}}(1-p)^{\frac{1}{\al}}}u^{\frac{-2}{\al}} r_\al\left(\frac{s+x}{2p^{\frac{1}{\al}}}u^{\frac{-1}{\al}}\right)r_\al\left(\frac{s-x}{2(1-p)^{\frac{1}{\al}}}u^{\frac{-1}{\al}}\right)du\\
&=\int_0^\infty \frac{\al}{2p^{\frac{1}{\al}}(1-p)^{\frac{1}{\al}}}z^{1-\al} r_\al\left(\frac{s+x}{2p^{\frac{1}{\al}}}z\right)r_\al\left(\frac{s-x}{2(1-p)^{\frac{1}{\al}}}z\right)dz.
\end{split}
\eqe
Finally we combine eqs. \eqref{xv_distribution}, \eqref{levy_measure_calculated}, \eqref{potential_density2} and integrate with respect to $dv$ to obtain
\eqb
\begin{split}
p_t(x)
%
%
&=\frac{\al}{\Gamma (1-\al)}\frac{1}{2p^\fal(1-p)^\fal}\\
&\quad\quad\quad\int_0^{t-|x|}\int_0^\infty v^{-\al}z^{1-\al} r_\al\left(\frac{t-v+x}{2p^{\frac{1}{\al}}}z\right)r_\al\left(\frac{t-v-x}{2(1-p)^{\frac{1}{\al}}}z\right)dzdv\\
&=\frac{\al}{\Gamma (1-\al)}\frac{1}{2p^\fal(1-p)^\fal}\\
&\quad\quad\quad\int_{|x|}^t\int_0^\infty\frac{z^{1-\al}}{(t-w)^{\al}}r_\al\left(\frac{w+x}{2p^\fal}z\right)r_\al\left(\frac{w-x}{2(1-p)^\fal}z\right)dzdw
\end{split}
\eqe
We used here the fact that $r_\al(x)$ vanishes outside the positive half-line and substituted $v=t-w$.
\end{proof}
One should mention here that the PDF of wait-first L\'evy walk in the extreme cases $p=0$ and $p=1$ was already derived in \cite{MSZ}.

For the special case $\al=\frac{1}{2}$ we get the following simple expression for $p_t(x)$:
\begin{corollary}
\label{alpha05}
When $\al=\frac{1}{2}$ and $p\in(0,1)$ the PDF $p_t(x)$ can be expressed as
\eqb
\begin{split}
p_t(x)&=\frac{\sqrt2}{\pi}p(1-p) \\
&\quad\quad\frac{\left(t-|x|\right)^{\frac{1}{2}}}{\left(2p^2t+(1-2p)(t+x)\right)\left(2p^2|x|+(1-2p)(x+|x|)\right)^{\frac{1}{2}}}\ind_{(0,t)}(|x|)
\end{split}
\eqe
\end{corollary}
%
%
\begin{proof}
We have (see \cite{Sato})
\eqb
r_{1/2}(x)=\frac{1}{2\sqrt{\pi}}x^{\frac{-3}{2}}\exp{\left(\frac{-1}{4x}\right)}.
\eqe
Taking into account this formula and using the property of Gamma function $\Gamma(3/2)=\frac{1}{2}\Gamma(1/2)=\frac{\sqrt{\pi}}{2}$ we can calculate the following integral:
\eqb
\label{integral_potential}
\begin{split}
\int_0^\infty& z^{\frac{1}{2}}r_{1/2}\left(\frac{w+x}{2p^2}z\right)r_{1/2}\left(\frac{w-x}{2(1-p)^2}z\right)dz\\
&=\frac{1}{4\pi}\left(\frac{(w+x)(w-x)}{4p^2(1-p)^2}\right)^{\frac{-3}{2}}\\
&\quad\quad\quad\int_0^\infty z^{\frac{-5}{2}} \exp{\left(-\frac{p^2(w-x)+(1-p)^2(w+x)}{2(w+x)(w-x)}\frac{1}{z}\right)}dz\\
&=\frac{2^\frac{3}{2}}{\sqrt{\pi}}\frac{p^3(1-p)^3}{\left(p^2(w-x)+(1-p)^2(w+x)\right)^\frac{3}{2}}.
\end{split}
\eqe
%
%
%
%
%
Substituting Eq. \eqref{integral_potential} into Eq.\eqref{T1_thesis} we arrive at
\eqb
\begin{split}
p_t(x)&=\frac{1}{\sqrt2 \pi}p(1-p)\\
&\int_{|x|}^t(t-w)^{-\frac{1}{2}}\left(\left((1-p)^2+p^2\right)w+\left((1-p)^2-p^2\right)x\right)^{-\frac{3}{2}}dw\ind_{(0,t)}(|x|).
\end{split}\eqe
One can notice that (below $a,b$ and $c$ are constants)
\eqb
\frac{d}{dw}\left(\frac{-2(c-w)^{\frac{1}{2}}}{(a+bc)(a+bw)^{\frac{1}{2}}}\right)=(c-w)^{-\frac{1}{2}}(a+bw)^{-\frac{3}{2}},
\eqe
which finally implies
\eqb
\begin{split}
p_t(x)&=\frac{\sqrt2}{\pi}p(1-p)\\
&\quad\quad \frac{\left(t-|x|\right)^{\frac{1}{2}}}{\left(2p^2t+(1-2p)(t+x)\right)\left(2p^2|x|+(1-2p)(x+|x|)\right)^{\frac{1}{2}}}\ind_{(0,t)}(|x|).
\end{split}\eqe
\end{proof}
%
%
It is worth to mention here, that in a special case $\al=\frac{1}{2}$ and $p=\frac{1}{2}$ we recover the known result (\cite{MSZ})
\eqb
p_t(x)=\frac{1}{2\pi}\int_{|x|}^t(t-w)^{-\frac{1}{2}}w^{-\frac{3}{2}}dw\ind_{(0,t)}(|x|)=\frac{(t-|x|)^{\frac{1}{2}}}{\pi t{|x|}^\frac{1}{2}}\ind_{(0,t)}(|x|).
\eqe
Figures \ref{fig:alpha05_p01} and \ref{fig:alpha05_different_p} present the densities obtained from the above Corollary.

%
%
\begin{figure}[!ht]
\centering
\includegraphics[scale=0.34]{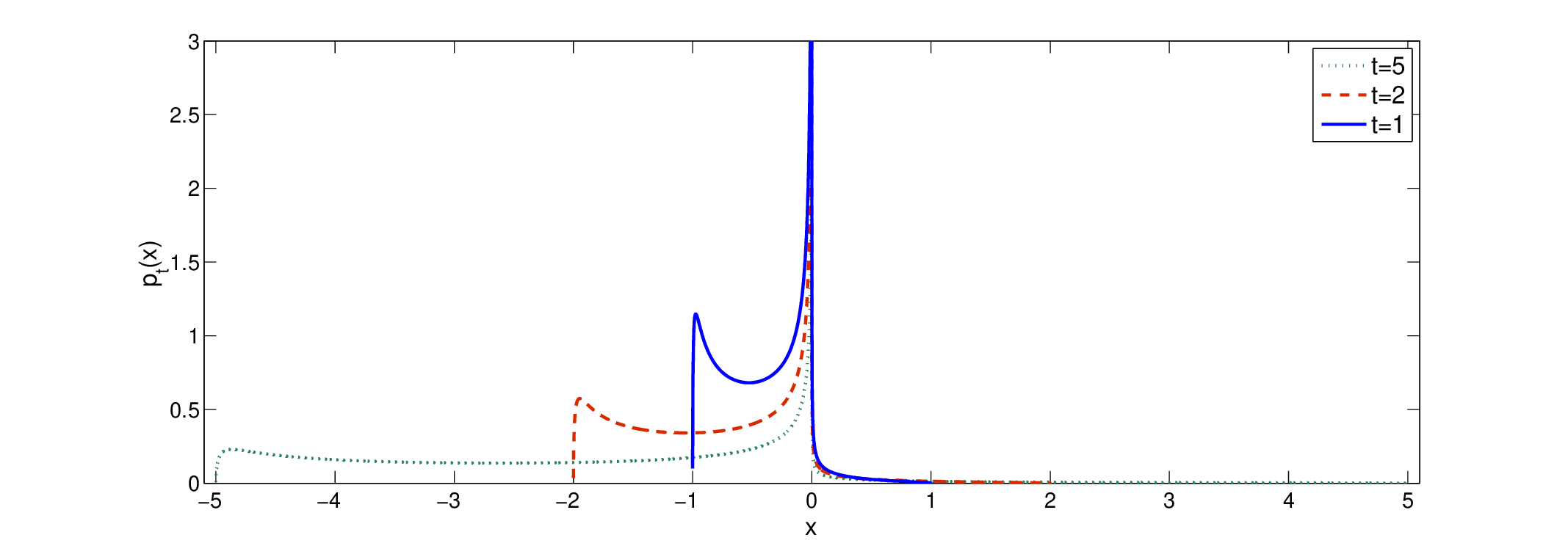}
\caption{
Plot of PDF $p_t(x)$ of the process $L_\al^-(S_\al^{-1}(t))$ for different values of $t$, $\alpha=0.5$, $p=0.1$.}
\label{fig:alpha05_p01}
\end{figure}
\begin{figure}[!ht]
\centering
\includegraphics[scale=0.34]{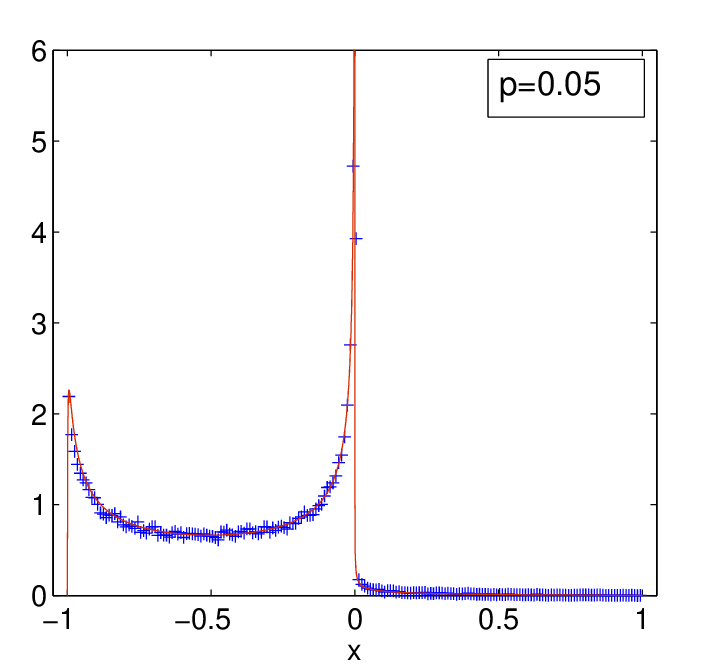}
\includegraphics[scale=0.34]{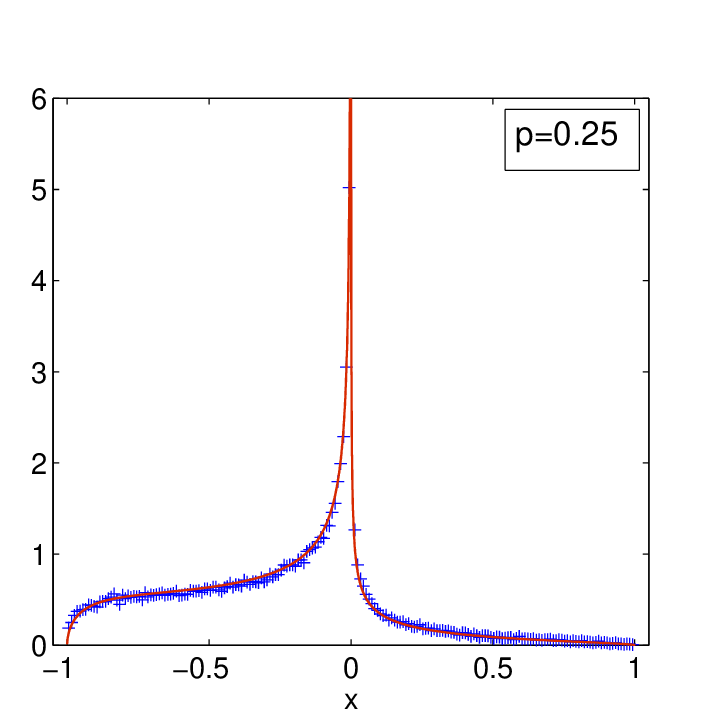}
\includegraphics[scale=0.34]{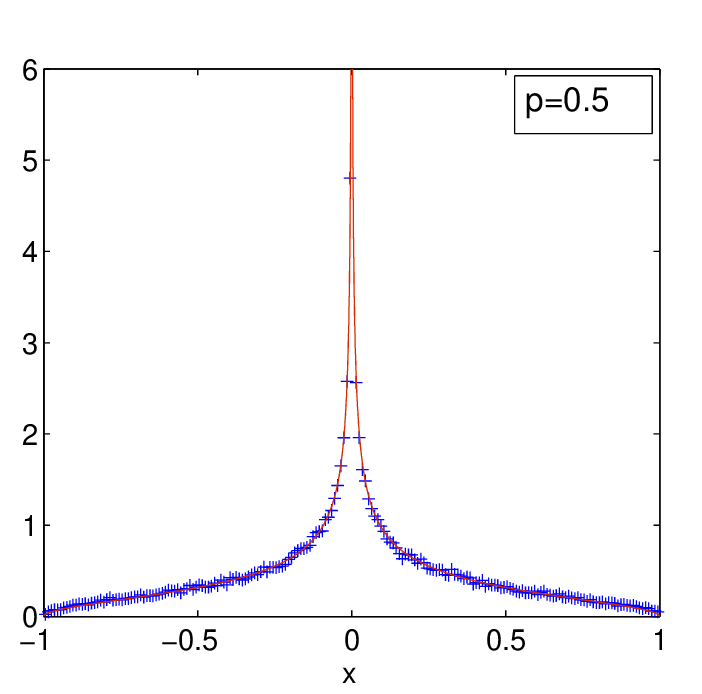}
\caption{
For $\al=0.5$ we compare the obtained densities $p_1(x)$ for different values of $p$ (red solid lines) with densities estimated using Monte Carlo methods (blue pluses).
}
\label{fig:alpha05_different_p}
\end{figure}
The drawback of using  Theorem \ref{T1_thesis} to calculate
$p_t(x)$ in the general case is the necessity of knowing values of
$r_\al(x)$. One can use some known algorithms to approximate
$r_\al(x)$, however results are not perfect. Computations are
time-consuming and inaccurate. They differ from those obtained via
Monte-Carlo methods. However, as the next Corollary shows, we can
express $p_t(x)$ in the form of an integral from Meijer G function
(see \cite{Prudnikov}). This special function is implemented in
most of numerical packages, including Mathematica and Matlab. This
representation is valid for rational $\al$ and can be used to
approximate irrational $\al$ with a rational one.
%
%
%
%
\begin{corollary}
\label{meijer_representation}
If $p\in(0,1)$ and $\al=\frac{l}{k}$, where $l,k\in \N$, then the PDF $p_t(x)$ of the process $X(t)=L_\al^-(S_\al^{-1}(t))$ equals
\eqb
\begin{split}
p_t(x)&=\frac{2^{1-\al-k+l}}{p\pi^{k-l}}\frac{k^2}{l^\al}\frac{\al}{\Gamma(1-\al)}
\int_{|x|}^t\Bigg[\frac{1}{(t-w)^\al}\frac{(w+x)^{\al-1}}{w-x}\\
&\quad\MeijerG*{k}{k}{l+k}{l+k}{-\Delta\left(k,k\left(\frac{\al}{l}-1\right)\right),\Delta(l,0)}{\Delta(k,0),-\Delta\left(l,l\left(\frac{\alpha}{l}-1\right)\right)}{\left(\frac{p}{1-p}\right)^\fal\left(\frac{w+x}{w-x}\right)^l}\Bigg]dw \ind_{(0,t)}(|x|),
\end{split}
\eqe
where $\MeijerG*{k_1}{k_2}{l_1}{l_2}{\Delta\left(k,a\right)}{\Delta(l,b)}{x}$ is the Meijer G function (see \cite{Prudnikov}) and
 $\Delta(k,a)=\frac{a}{k},\frac{a+1}{k},a+\frac{a+2}{k},...,\frac{a+k-1}{k}$ is a special list of $k$ elements.
\end{corollary}
%
%
\begin{proof}
For $\al=l/k$ where $k$, $l$ are positive integers with $k>l$ we can express $r_{\al}(x)$ in terms of the Meijer G function (see \cite{Penson}):
\eqb
r_{l/k}(x)=\frac{\sqrt{kl}}{(2\pi)^{(k-l)/2}}\frac{1}{x}\MeijerG*{k}{0}{l}{k}{\Delta(l,0)}{\Delta(k,0)}{\frac{l^l}{k^k}x^{-l}}.
\eqe
This gives us
\eqb
\begin{split}
\label{meijer_integral}
&\int_0^\infty z^{1-\al} q_\al\left(\frac{w+x}{2p^{\frac{1}{\al}}}z\right)q_\al\left(\frac{w-x}{2(1-p)^{\frac{1}{\al}}}z\right)dz\\
&\;=\frac{kl}{(2\pi)^{(k-l)}}\frac{4(1-p)^{\frac{1}{\al}}p^{\frac{1}{\al}}}{(w-x)(w+x)}\int_0^\infty \left[z^{-1-\al}
\MeijerG*{k}{0}{l}{k}{\Delta(l,0)}{\Delta(k,0)}{\frac{l^l}{k^k}\left(\frac{w+x}{2p^{\frac{1}{\al}}}z\right)^{-l}}
\right.
\\
&\left.
\quad\quad\quad\quad\quad\quad\quad\quad\quad\quad\quad\quad\quad\quad\MeijerG*{k}{0}{l}{k}{\Delta(l,0)}{\Delta(k,0)}{\frac{l^l}{k^k}\left(\frac{w-x}{2(1-p)^{\frac{1}{\al}}}z\right)^{-l}}\right]dz\\
&\;=\frac{k}{(2\pi)^{(k-l)}}\frac{4(1-p)^{\frac{1}{\al}}p^{\frac{1}{\al}}}{(w-x)(w+x)}\int_0^\infty \left[u^{{\frac{\alpha}{l}}-1}
\MeijerG*{k}{0}{l}{k}{\Delta(l,0)}{\Delta(k,0)}{\frac{l^l}{k^k}\left(\frac{w+x}{2p^{\frac{1}{\al}}}\right)^{-l}u}
\right.
\\
&\left.
\quad\quad\quad\quad\quad\quad\quad\quad\quad\quad\quad\quad\quad\quad
\MeijerG*{k}{0}{l}{k}{\Delta(l,0)}{\Delta(k,0)}{\frac{l^l}{k^k}\left(\frac{w-x}{2(1-p)^{\frac{1}{\al}}}\right)^{-l}u}\right]du\\
&\quad=\frac{k}{(2\pi)^{(k-l)}}\frac{4(1-p)^{\frac{1}{\al}}p^{\frac{1}{\al}}}{(w-x)(w+x)}\left(\frac{l^l}{k^k}\right)^{-\al/l}\left(\frac{w+x}{2p^\fal}\right)^\alpha\\
&\quad\quad\quad\quad\quad
\MeijerG*{k}{k}{l+k}{l+k}{-\Delta\left(k,k\left(\frac{\al}{l}-1\right)\right),\Delta(l,0)}{\Delta(k,0),-\Delta\left(l,l\left(\frac{\alpha}{l}-1\right)\right)}{\left(\frac{p}{1-p}\right)^\fal\left(\frac{w+x}{w-x}\right)^l}.
\end{split}
\eqe
We substituted in the above equations $z^{-l}=u$. Moreover we used properties of Meijer G functions (see \cite{Prudnikov}) to calculate the integral from multiplication of these functions. To end the proof we substitute the calculated above integral into Eq. \eqref{T1_thesis}.
\end{proof}
%
%
Figure \ref{fig:different_alpha_p25} presents densities calculated here.
Another method for calculating $r_\al$, where is not necessarily rational, was presented in \cite{Saa}.
%
%
\begin{figure}[!ht]
\centering
\includegraphics[scale=0.34]{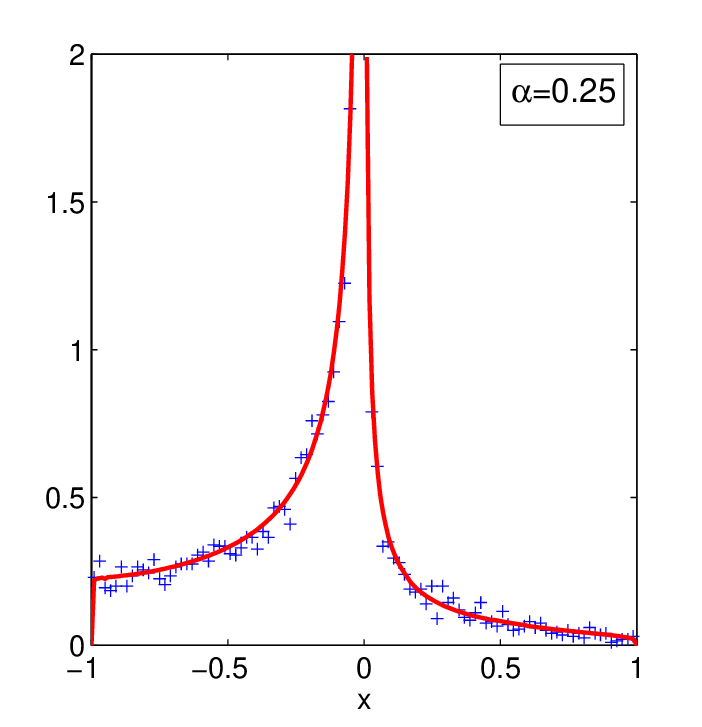}
\includegraphics[scale=0.34]{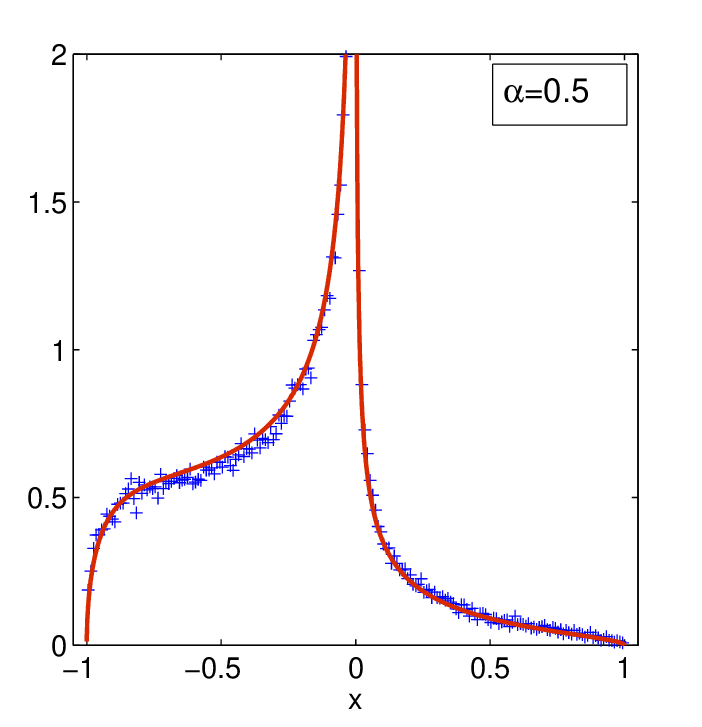}
\includegraphics[scale=0.34]{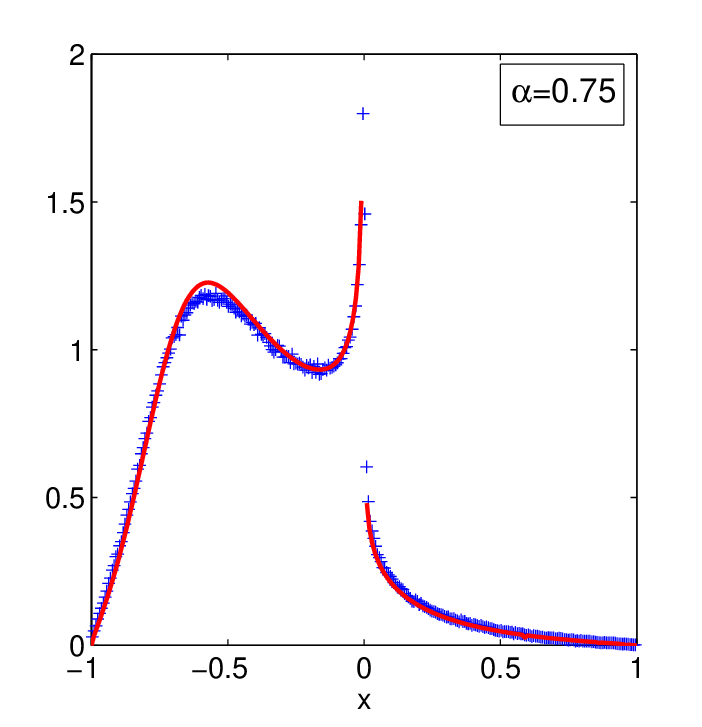}
\caption{ For $p=0.25$ we compare the densities $p_1(x)$ obtained
in Corollary \ref{meijer_representation} for different values of
$\al$ (red solid lines) with densities estimated using Monte Carlo
methods (blue pluses). } \label{fig:different_alpha_p25}
\end{figure}
%
%
%
%
\section{Densities of first-jump L\'evy walks}

This section is devoted to the first-jump scenario. Below we find
the PDF of the process $Y(t)$ from \eqref{limit2}. At the same
time this PDF solves the fractional equation \eqref{FDE2}.
\begin{theorem}\label{T2}
If $p\in(0,1)$ and $\al \in (0,1)$ then the PDF $w_t(y)$ of the
process $Y(t)=L_\al(S_\al^{-1}(t))$ equals:
\begin{itemize}
\item[(i)] if $|y|<t$, then
 \eqb \label{T2_thesis}
\begin{split}
w_t(y)&=\frac{\al^2 p^{1-\fal}}{2(1-p)^{\frac{1}{\al}}\Gamma (1-\al)}\int_{-t}^y
\int_{|x| \lor (t-y+x)}^t \int_0^\infty z^{1-\al}\\
&\quad\quad\quad\quad\quad\quad\quad  r_\al\left(\frac{w+x}{2p^{\frac{1}{\al}}}z\right)r_\al\left(\frac{w-x}{2(1-p)^{\frac{1}{\al}}}z\right)(y-x)^{-1-\al}dz
dwdx\\
&\;\;+\frac{\al^2 (1-p)^{1-\fal}}{2p^{\frac{1}{\al}}\Gamma (1-\al)}\int_{y}^t
\int_{|x| \lor (t+y-x)}^t\int_0^\infty z^{1-\al} \\
&\quad\quad\quad\quad\quad\quad\quad
r_\al\left(\frac{w+x}{2p^{\frac{1}{\al}}}z\right)r_\al\left(\frac{w-x}{2(1-p)^{\frac{1}{\al}}}z\right)(x-y)^{-1-\al}dz
dwdx,
\end{split}
\eqe

\item[(ii)] if $y\geq t$, then
\eqb \label{T2_thesis2}
\begin{split}
w_t(y)&=\frac{\al^2 p^{1-\fal}}{2(1-p)^{\frac{1}{\al}}\Gamma (1-\al)}\int_{-t}^t
\int_{|x|}^t \int_0^\infty z^{1-\al}\\
&\quad\quad\quad\quad\quad\quad\quad
r_\al\left(\frac{w+x}{2p^{\frac{1}{\al}}}z\right)r_\al\left(\frac{w-x}{2(1-p)^{\frac{1}{\al}}}z\right)(y-x)^{-1-\al}dz
dwdx,
\end{split}
\eqe

\item[(iii)] if $y\leq(-t)$, then
 \eqb \label{T2_thesis3}
\begin{split}
w_t(y)&=\frac{\al^2 (1-p)^{1-\fal}}{2p^{\frac{1}{\al}}\Gamma (1-\al)}\int_{-t}^t
\int_{|x|}^t \int_0^\infty z^{1-\al}\\
&\quad\quad\quad\quad\quad\quad\quad
r_\al\left(\frac{w+x}{2p^{\frac{1}{\al}}}z\right)r_\al\left(\frac{w-x}{2(1-p)^{\frac{1}{\al}}}z\right)(x-y)^{-1-\al}dz
dwdx.
\end{split}
\eqe
\end{itemize}
Here $r_\al(y)$ is a density of a positive $\al$-stable random variable $Z_\al$ with the Laplace transform
\eqb
\ex e^{-uZ_\al}=e^{-u^\al}
\eqe
and $a \lor b=\max{(a,b)}$
\end{theorem}
%
%
\begin{proof}
From Remark 4.2 in \cite{MS} we get
\eqb
\label{yr_distribution}
\pr \left(Y(t)=dy,R(t)=dr\right)=\int_{x\in \R}\int_{w\in[0,t]}U(dx,dw)\nu_{L_\al,S_\al}(dy-x,dr+t-w)
\eqe
where $R(t)$ is a process that counts how long the process $Y(t)$ remains constant from the moment $t$. After applying Eq.\eqref{levy_measure} the above equation reads
\eqb
\label{yr_distribution_calculated}
\begin{split}
&\pr \left(Y(t)=dy,R(t)=dr\right)\\
&=\int_{x\in \R}\int_{w\in[0,t]}U(dx,dw)\Big[p\delta_{dr+t-w}(dy-x)\nu_{S_\al}(dr+t-w)\\
&\quad\quad\quad\quad\quad\quad\quad\quad\quad+(1-p)\delta_{-dr-t+w}(dy-x)\nu_{S_\al}(dr+t-w)\Big]
\end{split}
\eqe
 Now we can recover the distribution of $Y(t)$ as a marginal distribution. We calculate the density $w_t(y)$ applying Tonelli's theorem to change the order of integration:
\eqb
\begin{split}
w_t(y)&=\frac{\al}{\Gamma(1-\al)}\\
&\quad\quad\int_{r\in [0,\infty)}
\int_{x\in \R}\int_{w\in[0,t]}
u(x,w)\Big[p\delta_{r+t-w}(y-x)(dr+t-w)^{-1-\al}\\
&\quad\quad\quad\quad\quad\quad\quad+(1-p)\delta_{-dr-t+w}(y-x)(dr+t-w)^{-1-\al}\Big]dwdxdr\\
&=\frac{\al}{\Gamma(1-\al)}\int_{-t}^{t}\int_{|x|}^t u(x,w)\Big[p\ind_{(-\infty,y)}(x)\ind_{(t-y+x,\infty)}(w) (y-x)^{-1-\al}\\
&\quad\quad\quad\quad\quad\quad\quad+(1-p)\ind_{(y,\infty)}(x)\ind_{(t+y-x,\infty)}(w) (x-y)^{-1-\al}\Big]dwdx.
\end{split}
\eqe
The area of integration in the above integral depends on the relation of $y$ with $t$ - we have three cases. The first case is described by the condition $\abs{y}<t$. Then
\eqb
\begin{split}
w_t(y)&=\frac{p\al}{\Gamma(1-\al)}\int_{-t}^{y}\int_{|x|\lor(t-y+x)}^t u(x,w)(y-x)^{-1-\al}dwdx\\
&\quad\quad+\frac{(1-p)\al}{\Gamma(1-\al)}\int_{y}^{t}\int_{|x|\lor(t+y-x)}^t u(x,w)(x-y)^{-1-\al}dwdx.
\end{split}
\eqe
If $y\geq t$ then
\eqb
w_t(y)=\frac{p\al}{\Gamma(1-\al)}\int_{-t}^{y}\int_{|x|\lor(t-y+x)}^t u(x,w)(y-x)^{-1-\al}dwdx
\eqe
and finally if $y\leq (-t)$ then
\eqb
w_t(y)=\frac{(1-p)\al}{\Gamma(1-\al)}\int_{y}^{t}\int_{|x|\lor(t+y-x)}^t u(x,w)(x-y)^{-1-\al}dwdx.
\eqe
Substituting Eq.\eqref{potential_density2} for $u(x,w)$ in all three cases ends the proof.
\end{proof}

In the special case $\al=\frac{1}{2}$ we get the following simple expression for $w_t(y)$:
\begin{corollary}
\label{o_corollary}
When $\al=\frac{1}{2}$ the PDF $w_t(y)$ can be expressed as
\eqb
\begin{split}
w_t(y)=\frac{p(1-p)t^\frac{1}{2}}{\pi}\frac{ t\left((t - y)^\frac{1}{2}-(t + y)^\frac{1}{2}  \right) +
       y\left((t - y)^\frac{1}{2} + (t + y)^\frac{1}{2}\right)}{ y(t - y)^\frac{1}{2}
     (t + y)^\frac{1}{2}\left(\left(1 - 2p(1 - p)\right)t +(1-2p) y \right)}
\end{split}
\eqe
if $|y|<t$ and as
\eqb
w_t(y)=\frac{p}{\pi}\frac{t^\frac{1}{2}}{ y \left(p(y-t)^\frac{1}{2} + (1-p)  (y+t)^\frac{1}{2}\right) }
\eqe
if $y\geq t$ and as
\eqb
w_t(y)=\frac{p-1}{\pi}\frac{t^\frac{1}{2}}{ y \left(p(t-y)^\frac{1}{2} + (1-p)  (-y-t)^\frac{1}{2}\right) }
\eqe
if $y\leq (-t)$.
\end{corollary}
%
%
\begin{proof}
The density $w_t(y)$ is symmetric in a following sense:
\eqb
w_t(y)_p=w_t(-y)_{1-p},
\eqe
where $w_t(y)_c$ denotes the PDF of the process $Y(t)$ with parameter $p=c$.
Thus it is enough to calculate $w_t(y)$ for $y>0$ - we assume now $y>0$. There are two cases: $y<t$ and $y\geq t$. In the first case
from Theorem \ref{T2} and Eq. \eqref{integral_potential} we have
\eqb
\spb
w_t(y)
&=\frac{p^2(1-p)}{2^{3/2}\pi}\int_{-t}^y
\int_{|x|\lor (t-y+x)}^t \\
&\quad\quad\quad\quad\quad\quad\quad\left(p^2(w-y)+(1-p)^2(w+y)\right)^{-3/2}(y-x)^{-3/2}dwdx\\
&\;+\frac{p(1-p)^2}{2^{3/2}\pi}\int_{y}^t
\int_{|x|\lor (t+y-x)}^t \\
&\quad\quad\quad\quad\quad\quad\quad{\left(p^2(w-y)+(1-p)^2(w+y)\right)^{-3/2}}(x-y)^{-3/2}dwdx
\end{split}\eqe
Notice that in the above integrals
$|x|\lor (t-y+x)=t-y+x$
when $x\geq\frac{y-t}{2}$ and
$|x|\lor (t-y+x)=|x|=-x$
if $x\leq\frac{y-t}{2}$.
Similarly $|x|\lor(t+y-x)=t+y-x$ when $x\leq\frac{y+t}{2}$ and
$|x|\lor(t+y-x)=|x|=x$ if $x\geq\frac{y+t}{2}$. Thus $w_t(y)$ can
be expressed as a sum of four integrals, each of them having a
simple region of integration.
We can calculate all of them using standard integration techniques.
After tedious computations we finally get the desired result for $w_t(y)$ in case when $y\in (0,t)$. In the second case, that is $y\geq t$, from Theorem \ref{T2} and Eq. \eqref{integral_potential} we have
\eqb
\begin{split}
w_t(y)
&=\frac{p^2(1-p)^2}{2^{3/2}\pi}\\
&\quad\Bigg(\int_{-t}^0
\int_{-x}^t {\left(p^2(w-y)+(1-p)^2(w+y)\right)^{-3/2}}(y-x)^{-3/2}dwdx\\
&\quad+\int_{0}^t
\int_{x}^t {\left(p^2(w-y)+(1-p)^2(w+y)\right)^{-3/2}}(y-x)^{-3/2}dwdx\Bigg)
\end{split}
\eqe
Again, applying standard integration techniques we obtain the desired result for $w_t(y)$ when $y>t$.
\end{proof}
%
%
In a special case $\al=\frac{1}{2}$ and $p=\frac{1}{2}$ the PDF $w_t(y)$ has the form
\eqb
\begin{split}
w_t(y)=\frac{1}{2\pi}\frac{t^\frac{1}{2}\left(t\left((t + y)^\frac{1}{2}-(t - y)^\frac{1}{2}  \right) -
       y\left((t - y)^\frac{1}{2} + (t + y)^\frac{1}{2}\right)\right)}{ y(t - y)^\frac{1}{2}
     (t + y)^\frac{1}{2}t}
\end{split}
\eqe
if $|y|<t$ and
\eqb
w_t(y)=\frac{1}{\pi}\frac{t^\frac{1}{2}}{ |y| \left(|y-t|^\frac{1}{2} +   |y+t|^\frac{1}{2}\right) }
\eqe
if $|y|\geq t$.
Figure \ref{fig:o_alpha05_different_p} presents $w_1(y)$ for different values of $p$. Notice that in all cases we have a characteristic sharp peak at $y=t$ and $y=-t$. However when $p\rightarrow 0$ the left peak  gets bigger and the right one diminishes. The opposite situation appears when $p \rightarrow 1$.
%
%
\begin{figure}[H]
\centering
\includegraphics[scale=0.34]{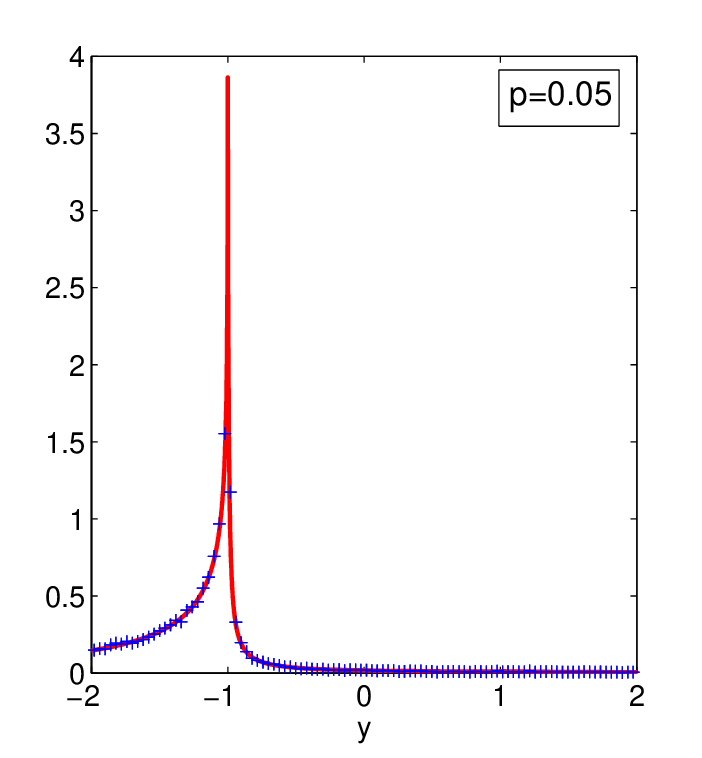}
\includegraphics[scale=0.34]{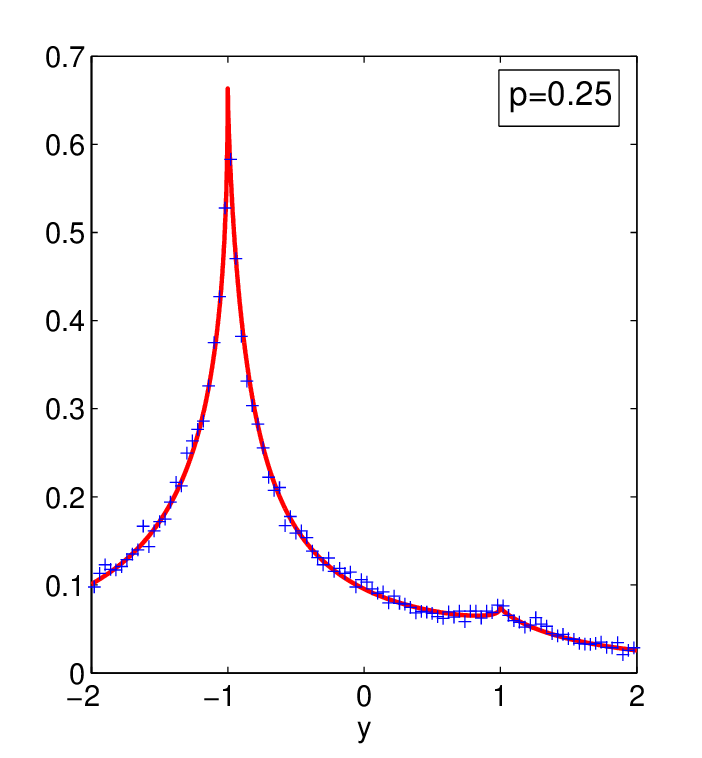}
\includegraphics[scale=0.34]{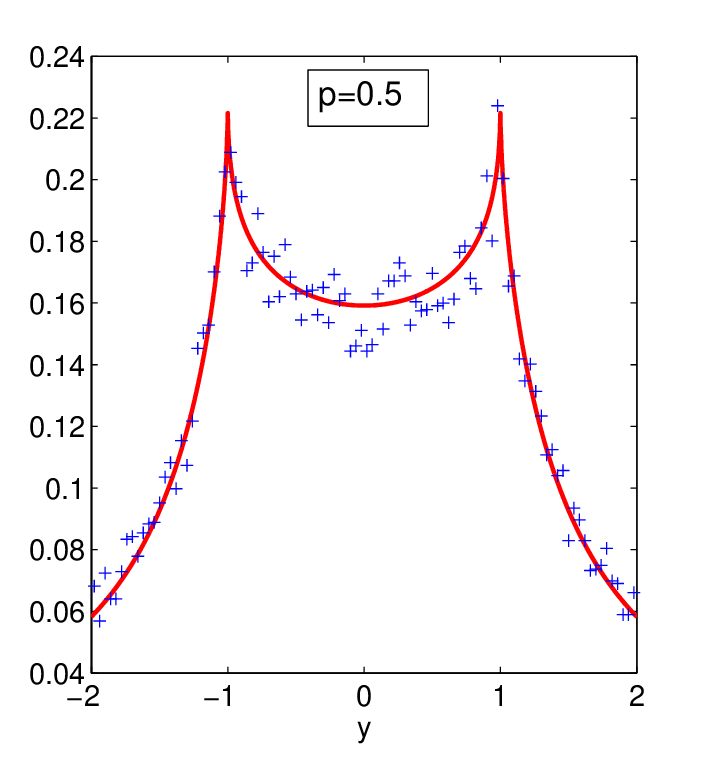}
\caption{ For $\al=0.5$ we compare the densities $w_1(y)$ obtained
in Corollary \ref{o_corollary} for different values of $p$ (red
solid lines) with densities estimated using Monte Carlo methods
(blue pluses). } \label{fig:o_alpha05_different_p}
\end{figure}
\begin{figure}[H]
\label{fig:o_alpha05_p05}
\centering
\includegraphics[scale=0.34]{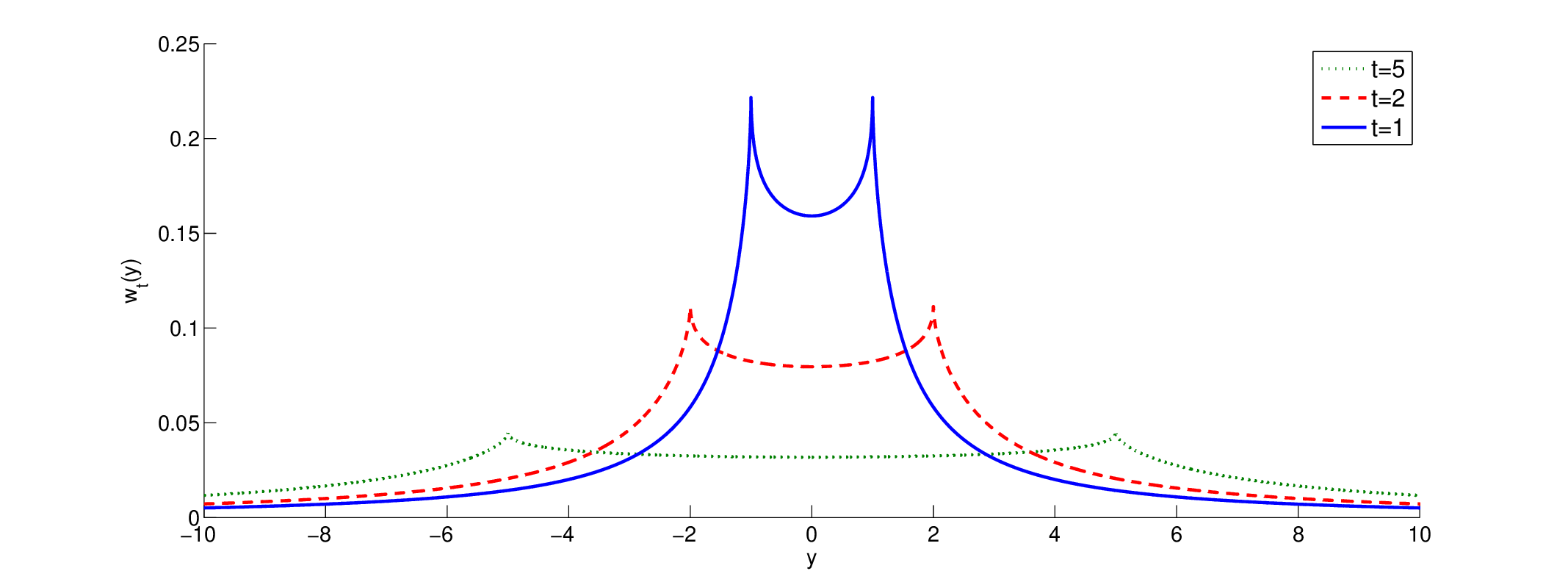}
\caption{
Plot of PDF $w_t(y)$ of the process $L_\al(S_\al^{-1}(t))$ for different values of $t$, $\alpha=0.5$, $p=0.5$.}
\end{figure}

\section*{Acknowledgements}
This research was partially supported by NCN Maestro grant no. 2012/06/A/ST1/00258.



 \bigskip \smallskip

 \it

 \noindent
$^1$
Hugo Steinhaus Center, \\
Department of Mathematics, \\
Wroclaw University of Technology, \\
Wyspianskiego 27, 50-370 Wroclaw, Poland. \\[4pt]
e-mail: marcin.magdziarz@pwr.edu.pl

\end{document}